\documentclass[a4wider]{amsart}
\usepackage{enumerate}
\usepackage{fontenc}
\usepackage{color}
\usepackage{amsfonts,amsmath,amssymb,amsthm}
\usepackage{dsfont}
\usepackage{latexsym}
\usepackage{enumerate}
\usepackage[T1]{fontenc}
\usepackage{yfonts}
\usepackage{amsmath,amsfonts,amstext}
\usepackage{amssymb,amsmath,amscd,graphicx,epsfig,hhline,mathrsfs,yhmath,latexsym,url}
\usepackage[arrow, matrix, curve]{xy}
\usepackage[british]{babel}
\usepackage{paralist}               	
\usepackage[parfill]{parskip}			

\newtheorem{theorem}{Theorem}[section]

\newtheorem{observation}[theorem]{Observation}

\newtheorem{conjecture}[theorem]{Conjecture}

\newtheorem{lemma}[theorem]{Lemma}

\newtheorem{remark}[theorem]{Remark}

\newtheorem*{T*}{Theorem}
\newtheorem*{A*}{Proposition}
\newtheorem*{Cor*}{Corollary}

\theoremstyle{definitionbreak}
\theoremstyle{definitionbreak}\newtheorem*{D*}{Definition}
\newtheorem{Ex}[theorem]{Example}

\newtheorem{prelem}{{\bf Theorem}}

\title{On Sidorenko's conjecture for determinants and Gaussian Markov random fields}
\author{Bal\'azs Szegedy}

\begin{document}		
\maketitle

\begin{abstract} We study a class of determinant inequalities that are closely related to Sidorenko's famous conjecture (Also conjectured by Erd\H os and Simonovits in a different form). Our results can also be interpreted as entropy inequalities for Gaussian Markov random fields (GMRF). We call a GMRF on a finite graph $G$ homogeneous if the marginal distributions on the edges are all identical. We show that if $G$ satisfies Sidorenko's conjecture then the differential entropy of any homogeneous GMRF on $G$ is at least $|E(G)|$ times the edge entropy plus $|V(G)|-2|E(G)|$ times the point entropy. We also prove this inequality in a large class of graphs for which Sidorenko's conjecture is not verified including the so-called M\"obius ladder: $K_{5,5}\setminus C_{10}$. The connection between Sidorenko's conjecture and GMRF's is established via a large deviation principle on high dimensional spheres combined with graph limit theory. 
\end{abstract}

\section{Introduction}

For a finite graph $G=(V,E)$ and $x\in (-1,1)$ let $\Psi(G,x)$ denote the set of $V\times V$ matrices $M$ such that
\begin{enumerate}
\item $M$ is positive definite,
\item Every diagonal entry of $M$ is $1$,
\item $M_{i,j}=x$ for every edge $(i,j)$ of $G$.
\end{enumerate}

The strict concavity of the function $M\mapsto \log(\det(M))$ and the convexity of $\Psi(G,x)$ together imply that there is a unique matrix $\Sigma(G,x)$ in $\Psi(G,x)$ which maximizes determinant. For probabilists, this matrix is know as the covariance matrix of the Gaussian Markov random field $\{X_v\}_{v\in V(G)}$ (or shortly GMRF) on $G$ in which $X_v\sim N(0,1)$ holds for every vertex $v$ and $\mathbb{E}(X_iX_j)=x$ holds for every edge $(i,j)$ of $G$.
The function $\tau(G,x):=\det(\Sigma(G,x))$ is an interesting analytic function of $x$ for every fixed graph $G$. One can for example easily see that if $G$ is a tree then $$\tau(G,x)=(1-x^2)^{|E(G)|}$$ and if $G$ is the four cycle then
$$\tau\Bigl(~C_4,\sqrt{X^2/2+x/2}~\Bigr)=1-2x+2x^3-x^4.$$
In general we are not aware of any nice explicit formula for $\tau(G,x)$ however we know that the power series expansion of $\tau(G,x)$ around $0$ has integer coefficients (for the sketch of the proof see chapter \ref{alg}) that carry interesting combinatorial meaning which will be studied in a separate paper. Based on extensive computer experiments (using the algorithm described in chapter \ref{alg}) we conjecture the following.


\begin{conjecture}\label{conj1} If $G$ is any graph and $x\in [0,1)$ then $\tau(G,x)\geq(1-x^2)^{|E(G)|}$. 
\end{conjecture}

Note that since $\tau(e,x)=1-x^2$ for the single edge $e$ the conjecture says that $\tau(G,x)\geq\tau(e,x)^{|E(G)|}$.
It is not hard to verify conjecture \ref{conj1} for complete graphs and cycles. 

\begin{remark}\label{remloc} With some effort one can compute the first non zero coefficient in the power series of $\ln(\tau(G,x))-|E(G)|\ln(1-x^2)$ around $0$ and find that it is positive. This establishes a local version conjecture \ref{conj1} in a small interval $[0,\epsilon)$.
\end{remark}

It is easy to see that if $G$ is bipartite then $\tau(G,x)$ is an even function of $x$ and thus conjecture \ref{conj1} implies the following weaker conjecture.

\begin{conjecture}\label{conj2} If $G$ is any bipartite graph and $x\in (-1,1)$ then $\tau(G,x)$\\ $\geq (1-x^2)^{|E(G)|}$.
\end{conjecture}

Using a large deviation principle on high dimensional spheres and logarithmic graph limits \cite{Sz1} we will relate the quantities $\tau(G,x)$ to more familiar subgraph densities $t(G,H)$ from extremal combinatorics. If $G$ and $H$ are finite graphs then $t(G,H)$ denotes the probability that a random map from $V(G)$ to $V(H)$ takes edges to edges. We say that $G$ is a {\bf Sidorenko graph} if $t(G,H)\geq t(e,H)^{|E(H)|}$ holds for every graph $H$. Sidorenko's famous conjecture \cite{Sid} says that every bipartite graph is a Sidorenko graph. Even though this conjecture is still open there are many known examples for Sidorenko graphs including rather general infinite families. For literature on Sidorenko's conjecture see: \cite{BP},\cite{Sid},\cite{Hat},\cite{L},\cite{LS},\cite{BR},\cite{KLL},\cite{CFS},\cite{Sz1},\cite{Sz2}. Note that a somewhat stronger version of the conjecture was formulated by Erd\H os and Simonovits in \cite{Sim}. 
The smallest graph for which Sidorenko's conjecture is not known is the so-called M\"obius ladder which is $K_{5,5}\setminus C_{10}$. Our next theorem verifies conjecture \ref{conj2} for all Sidorenko graphs.

\begin{theorem}\label{thm1} If $G$ is a Sidorenko graph then $G$ satisfies conjecture \ref{conj2}.
\end{theorem}

A similar theorem for finite state Markov random fields was proved by the author in \cite{Sz1}. The proof of theorem \ref{thm1} does not follow directly from the results in \cite{Sz1}, it relies on the above mentioned large deviation principle on high dimensional spheres (see theorem \ref{sphere} and theorem \ref{ldsph}) which is interesting on its own right. Using methods developed in \cite{Sz2} we also prove the next theorem.

\begin{theorem}\label{thm2} Let $G$ be a bipartite graph with color classes $V_1$ and $V_2$. If
$$\sum_{v\in V_2}  {\rm deg}(v)({\rm deg}(v)-1)\geq |V_1|(|V_1|-1)$$ then $G$ satisfies conjecture \ref{conj2}.
\end{theorem}

Theorem \ref{thm1} verifies conjecture \ref{conj2} for the M\"obius ladder but also for large classes of graphs that are not known to be Sidorenko. In particular theorem \ref{thm2} shows that possible counter examples to conjecture \ref{conj2} have to be quite sparse.

The importance of the function $\tau(G,x)$ is also rooted in the fact that the differential entropy of the GMRF corresponding to $\Sigma(G,x)$ is $$\frac{|V|}{2}\ln(2\pi e)+\frac{1}{2}\ln(\tau(G,x)).$$
The next observation connects the theorems in this paper to differential entropy. Let us denote the differential entropy of a joint distribution $\{X_i\}_{i\in I}$ by $\mathbb{D}(\{X_i\}_{i\in I})$.

\begin{observation}\label{obs1} If $G=(V,E)$ satisfies conjecture \ref{conj2} then every homogeneous GMRF $\{X_v\}_{v\in V}$ on $G$ satisfies the next entropy inequality
\begin{equation}\label{entin}
\mathbb{D}(\{X_v\}_{v\in V})-\sum_{(i,j)\in E}\mathbb{D}(\{X_i,X_j\})+\sum_{v \in V} ({\rm deg}(v)-1)\mathbb{D}(X_v)\geq 0.
\end{equation}
\end{observation}

Using the homogeneity of $\{X_v\}_{v\in V}$, the inequality in the observation is equivalent with the fact that the differential entropy of the whole field is at lest $|E(G)|$ times the edge entropy plus $|V(G)|-2|E(G)|$ times the point entropy. Formally, the point entropies are only needed to cancel the extra additive constant in the formula for differential entropy however we believe that they may become important in a mere general circle of questions. The left hand side of (\ref{entin}) is an interesting invariant for general GMRF's where the marginals are not necessarily equal. 

\section{A large devation principle on the sphere}

It is well known that if $k\in\mathbb{N}$ is a fixed number and $n$ is big compared to $k$ then if we choose independent uniform vectors $v_1,v_2,\dots,v_k$ in the sphere $S_{n-1}=\{x|x\in\mathbb{R}^n, \|x\|_2=1\}$ then with probability close to one the vectors are close to be pairwise orthogonal. It will be important for us to estimate the probability of the atypical event that the scalar product matrix $(v_i,v_j)_{1\leq i,j\leq k}$ is close to some matrix $A$ that is separated from the identity matrix. Let $\lambda_k$ denote the Lebesgue measure on the space of symmetric $k\times k$ matrices with $1's$ in the diagonal. In this chapter we give a simple formula for the density function $(v_i,v_j)_{1\leq i,j\leq k}$ relative to the Lebesgue measure $\lambda_k$. Using this formula we prove a large deviation principle for the scalar product matrices of random vectors.

\begin{theorem}\label{sphere} Assume that $n\geq k\geq 2$ are integres. Let $v_1,v_2,\dots,v_k$ be independent, uniform random elements on the sphere $S_{n-1}$ and let $M(k,n)$ be the $k\times k$ matrix with entries $M(k,n)_{i,j}:=(v_i,v_j)$. The probability density function $f_{k,n}$ of $M(k,n)$ is supported on the set $\mathcal{M}_k$ of positive semidefinte $k\times k$ matrices with $1's$ in the diagonal entries and is given by the formula $$f_{k,n}(M)=\det(M)^{(n-k-1)/2}\Gamma(n/2)^k\Gamma_k(n/2)^{-1}$$ where $\Gamma_k$ is the multivariate $\Gamma$-function.
\end{theorem}

\begin{proof} Let $\{X_i\}_{i=1}^k$ be a system of $k$ independent $\chi_n$ distributions.
Let $M'(k,n)$ be the $k\times k$ matrix with entries $M'(k,n)_{i,j}:=(X_iv_i,X_jv_j)=X_iX_j(v_i,v_j)$. We have that $M'(k,n)_{i,i}=X_i^2$ holds for $1\leq i\leq k$. The definition of the $\chi_n$ distribution and the spherical symmetry of the $n$ dimensional standard normal distribution imply that $X_iv_i$ is an $n$ dimensional standard normal distribution. We obtain that the distribution of $M'(k,n)$ is the Wishart distribution corresponding to the $k\times k$ identity matrix . It follows that the density function $\tilde{f}_{k,n}$ of $M'(k,n)$ is supported on positive semidefinite matrices and is given by $$\tilde{f}_{k,n}(M)=\det(M)^{(n-k-1)/2}e^{-{\rm tr}(M)/2}2^{-kn/2}\Gamma_k(n/2)^{-1}.$$
The next step is to compute the conditional distribution of $M'(k,n)$ in the set $M'(k,n)_{i,i}=X_i^2=1$ for $1\leq i\leq k$. Using the fact that the density function $g_n(x)$ of $\chi_n^2$ is
$$g_n(x)=x^{n/2-1}e^{-x/2}2^{-n/2}\Gamma(n/2)^{-1}$$
the statement of the proposition follows from $f_{k,n}(M)=\tilde{f}_{k,n}(M)/g_n(1)^k$ for $M\in\mathcal{M}_k$ and $f_{k,n}(M)=0$ for $M\notin\mathcal{M}_k$. 
\end{proof}

\begin{remark} It is a nice fact that theorem \ref{sphere} allows us to give an explicit formula for the volume of the spectahedron $\mathcal{M}_k$. If $n=k+1$ then $f_{k,n}$ is a constant function and by the fact that it is a density function, this constant is the inverse volume of $\mathcal{M}_k$. We obtain that ${\rm Vol}(\mathcal{M}_k)=\Gamma((k+1)/2)^{-k}\Gamma_k((k+1)/2)$.
\end{remark}

\begin{lemma}\label{asym} For $k\geq 2$ we have that $$\lim_{n\to\infty}\frac{\Gamma(n/2)^k\Gamma_k(n/2)^{-1}}{(n/(2\pi))^{k(k-1)/4}}=1.$$
\end{lemma}

\begin{proof}
It is straightforward from the furmulas that
$$\Gamma(n/2)^k\Gamma_k(n/2)^{-1}=c_n^{k-1}c_{n-1}^{k-2}\dots c_{n-k+2}$$ where $c_r=\pi^{-1/2}\Gamma(r/2)/\Gamma((r-1)/2)$. It is well known that $\lim_{r\to\infty}\Gamma(r)\Gamma(r-\alpha)^{-1}r^{-\alpha}=1$. It follows that $\lim_{r\to\infty} c_r(r/(2\pi))^{-1/2}=1$ which completes the proof.
\end{proof}

Now we are ready to formulate and prove our large deviation principle. Let us denote by $\mu_{k,n}$ the probability measure corresponding to the random matrix model $M(k,n)$ defined in theorem \ref{sphere}. We have that $\mu_{k,n}$ is concentrated on the closed set $\mathcal{M}_k$. If $n\geq k\geq 2$ then $\mathcal{M}_k$ is a compact convex set of positive measure in the space of symmetric $k\times k$ matrices with ones in the diagonal. For a measurable function $f:\mathcal{M}_k\rightarrow\mathbb{R}$ we denote by $\|f\|_\infty$ the essential maximum of $f$ relative to the measure $\lambda_k$. Note that $\|f\|_\infty$ can differ from $\sup_{x\in\mathcal{M}_k}f(x)$ because changes in $f$ on $0$ measure sets are ignored. In general will use the norms $\|.\|_p$ for functions on $\mathcal{M}_k$.

\begin{theorem}[Large deviation principle on the sphere]\label{ldsph} Let $k\geq 2$ be a fixed integer. Let $A\subseteq\mathcal{M}_k$ be a Borel measurable set. We have that
$$ \lim_{n\to\infty} n^{-1}\ln(\mu_{n,k}(A))=(1/2)\ln\|1_A\det\|_\infty.$$
\end{theorem}

\begin{proof} We have from theorem \ref{sphere} that
$$\mu_{k,n}(A)=c_{k,n}\int_{A}(\det)^{n-k-1} d\lambda_k=c_{k,n}\int_{\mathcal{M}_k}1_A(\det)^{n-k-1} d\lambda_k=$$
$$c_{k,n}\int_{\mathcal{M}_k}(1_A\det)^{n-k-1} d\lambda_k=c_{k,n}\|1_A\det\|_{(n-k-1)/2}^{(n-k-1)/2}$$ where $c_{k,n}=\Gamma(n/2)^k\Gamma_k(n/2)^{-1}$.
It follows that
$$n^{-1}\ln(\mu_{k,n}(A))=\ln(c_{k,n}^{1/n})+\frac{n+k-1}{2n}\ln\|1_A\det\|_{(n-k-1)/2}.$$
From lemma \ref{asym} we get that
$$\lim_{n\to\infty} \ln(c_{k,n}^{1/n})=0.$$
Now the statement of the theorem follows from $\lim_{p\to\infty}\|1_A\det\|_p=\|1_A\det\|_\infty$.
\end{proof}

\section{Spherical graphons and the proof of theorem \ref{thm1}}

A {\bf graphon} (see \cite{LSz}) is symmetric measurable function of the form $W:\Omega^2\rightarrow [0,1]$ where $(\Omega,\mu)$ is a standard probability space. If $G$ is a finite graph then it makes sense to introduce the "density" of $G$ in $W$ using the formula 
$$t(G,W)=\int_{x\in\Omega^{V(G)}}\prod_{(i,j)\in E(G)}W(x_i,x_j)~ d\mu^k.$$ 
Note that the conjecture of Sideronko was originally stated in this integral setting and it says that $t(G,W)\geq t(e,W)^{|E(G)|}$ holds for every bipartite graph $G$ and graphon $W$.

In this chapter we prove theorem \ref{thm1} using special graphons that we call {\bf spherical graphons}.
Let $A\subseteq [-1,1]$ be a Borel measurable set and let $n$ be a natural number. Let us define the graphon ${\rm Sph}_{A,n}:S_n\times S_n\rightarrow [0,1]$ such that
${\rm Sph}_{A,n}(x,y)=1$ if $(x,y)\in A$ and ${\rm Sph}_{A,n}(x,y)=0$ if $(x,y)\notin A$.

For a Borel measurable set $A\subseteq [-1,1]$ and graph $G$ let $\Psi(G,A)$ denote the set of positive semidefinite $V(G)\times V(G)$ matrices $M$ such that the diagonal entries of $M$ are all $1's$ and $M_{i,j}\in A$ holds for every $(i,j)\in E(G)$. It is clear that using the notation from the previous chapter we have that
$$ t(G,{\rm Sph}(A,n-1))=\mu_{|V(G)|,n}(\Psi(G,A)) $$
It follows from theorem \ref{ldsph} that
\begin{equation}\label{limitdens}
\lim_{n\to\infty} n^{-1}\ln(t(G,{\rm Sph}(A,n))=(1/2)\ln\|1_{\Psi(G,A)}\det\|_\infty.
\end{equation}
Now we get the next lemma.
\begin{lemma}\label{denslem} Assume that $A\subseteq [-1,1]$ is a Borel set and $G$ is a Sidorenko graph.
Then
$$\|1_{\Psi(G,A)}\det\|_\infty\geq \|1_{\Psi(e,A)}\|_\infty^{|E(G)|}.$$
\end{lemma}

\begin{proof} 
The Sidorenko property of $G$ implies that for every $n$ we have that
$$n^{-1}\ln(t(G,{\rm Sph}(A,n))\geq |E(G)|n^{-1}\ln(t(e,{\rm Sph}(A,n)).$$
Then (\ref{limitdens}) completes the proof by taking the limit $n\to\infty$.
\end{proof}

To prove theorem \ref{thm1} let $x\in (-1,1)$ arbitrary and let $A_\epsilon:=[x-\epsilon,x+\epsilon]\cap (-1,1)$.
It follows from the continuity of determinants that $\lim_{\epsilon\to 0}\|1_{\Psi(G,A_\epsilon)}\det\|_\infty=\tau(G,x)$ holds for every graph $G$.
Then lemma \ref{denslem} complets the proof.

\section{Conditional independent couplings and the proof of theorem \ref{thm2}}

In the proof of theorem \ref{thm2} we will use a gluing operation for positive definite matrices that corresponds to conditional independent couplings of Gaussian distributions in the probabilistic setting.

\begin{lemma}\label{linalg}
Assume that $X$ and $Y$ are two finite sets with $X\cap Y=Z$ and $X\cup Y=Q$. Assume furthermore $A\in\mathbb{R}^{X\times X}$ and $B\in\mathbb{R}^{Y\times Y}$ are two positive definite matrices such that their submatrices  $A_{Z\times Z}$ and $B_{Z\times Z}$ are equal to some matrix $C\in\mathbb{R}^{Z\times Z}$. Let $\tilde{A},\tilde{B}$ and $\tilde{C}$ be the matrices in $\mathbb{R}^{Q\times Q}$ obtained from $A^{-1},B^{-1}$ and $C^{-1}$ by putting zeros to the remaining entries. Then the matrix
$$D:=(\tilde{A}+\tilde{B}-\tilde{C})^{-1}$$ satisfies the following conditions.
\begin{enumerate}
\item $D_{X\times X}=A$ , $D_{Y\times Y}=B$,
\item $D$ is positive definite
\item $\det(D)=\det(A)\det(B)\det(C)^{-1}$.
\end{enumerate}

\begin{proof} The statement can be checked with elementary linear algebraic methods. To highlight the connection to probability theory we give the probabilistic proof which is also more elegant. We can regard $A,B$ and $C$ as covariance matrices of Gaussian distributions $\mu_A,\mu_B$ and $\mu_C$ on $\mathbb{R}^X,\mathbb{R}^Y$ and $\mathbb{R}^Z$ with density functions $f_A,f_B$ and $f_C$. The condition $A_{Z\times Z}=B_{Z\times Z}$ is equivalent with the fact that the marginal distribution of both $\mu_A$ and $\mu_B$ on $\mathbb{R}^Z$ is equal to $\mu_C$. The conditional independent coupling of $\mu_A$ and $\mu_B$ over the marginal $\mu_C$ has density function
$$f(v)=(2\pi)^{-|Q|/2}e^{-\frac{1}{2}(vP_XA^{-1}P_Xv^T+vP_YB^{-1}P_Yv^T-vP_ZC^{-1}P_Zv^T)}$$
where $P_X,P_Y$ and $P_Z$ are the projections to the coordiantes in $X,Y$ and $Z$.
We have by the definition of $D$ that
$$f(v)=(2\pi)^{-|Q|/2}e^{-\frac{1}{2} vD^{-1}v^T}$$
and thus $f$ is the density function of the Gaussian distribution $\mu$ with covariance matrix $D$.
The first property of $D$ follows from that fact that the marginals of $\mu$ on $X$ and $Y$ are have covariance matrices $A$ and $B$.
The second property of $D$ follows from the fact that it is a covariance matrix of a non-degenerated Gaussian distribution. The third property follows from the fact that
$$0=\mathbb{D}(\mu)-\mathbb{D}(\mu_A)-\mathbb{D}(\mu_B)+\mathbb{D}(\mu_C)$$ holds for the differential entropies in a conditionally independent coupling. On the other hand the right hand side is equal to
$$\ln(\det(D))-\ln(\det(A))-\ln(\det(B))+\ln(\det(C)).$$
\end{proof}

\end{lemma}

We will refer to the matrix $D$ and the {conditionally independent coupling of $A$ and $B$ (over $C$) and we denote it by $A\curlyvee B$.

Let $G$ be a bipartite graph on color classes $V_1$ and $V_2$ such that $V_1\sqcup V_2=V$. Let $x\in (-1,1)$ be an arbitrary number. Our goal in this chapter is to build up a matrix $M$ in $\Psi(G,x)$ using a sequence of conditionally independent couplings. Then we show that if $G$ satisfies the degree condition in theorem \ref{thm2} then $\det(M)\geq (1-x^2)^{|E(G)|}$. Thus we construct a witness matrix for the fact that $\tau(G,x)\geq (1-x^2)^{|E(G)|}$.

Let $M_0$ be the $V_1\times V_1$ matrix with $1$'s in the diagonal and $x^2$ elsewhere. For $v\in V_2$ let $N(v)\subseteq V_1$ denote the set of neighbors of $v$ and let $M^v$ be the $(\{v\}\cup N(v))\times(\{v\}\cup N(v))$ matrix that has $1's$ in the diagonal and $M^v_{v,w}=M^v_{w,v}=x$ for every $w\in N(v)$ furthermore $M^v_{i,j}=x^2$ for $i,j\in N(v)$ with $i\neq j$. We have the followings:
\begin{enumerate}
\item $\det(M_0)=(1-x^2)^{a-1}(1+(a-1)x^2)$,
\item $\det(M^v)=(1-x^2)^{d_v}$,
\item $\det(M^v_{N(v)\times N(v)})=(1-x^2)^{d_v-1}(1+(d_v-1)x^2)$.
\end{enumerate} 
where $a=|V_1|$ and $d_v=|N(v)|$.
Assume that $V_2=\{1,2,\dots,b\}$. Let
$$M:=(\dots((M_0\curlyvee M^1)\curlyvee M^2)\dots)\curlyvee M^b$$
be the conditionally independet coupling of the matrices $M_0,M^1,M^2,\dots,M^b$.
Using all the previous formulas, lemma \ref{linalg} and the fact that $M\in\Psi(G,x)$ we obtain that
\begin{equation}\label{lbound}
\tau(G,x)\geq \det(M)=(1-x^2)^{a+b-1}(1+(a-1)x^2)\prod_{i=1}^b(1+(d_i-1)x^2)^{-1}.
\end{equation}

For the proof of theorem \ref{thm2} it remains to show the next lemma.

\begin{lemma} If $a,b,d_i$ and $x$ are as above and $\sum_{i=1}^b d_i(d_i-1)\geq a(a-1)$ then the right hand side of (\ref{lbound}) is at least $(1-x^2)^{|E(G)|}$.
\end{lemma}

\begin{proof} By using that $\sum_{i=1}^b d_i=|E(G)|$ we can simplify the desired inequality to
$$(1+(a-1)y)(1-y)^{a-1}\geq\prod_{i=1}^b (1+(d_i-1)y)(1-y)^{d_i-1}$$ where $y=x^2$. For $y=0$ both sides are equal to $1$. Now it is enough to show that the logarithmic derivative of the left hand side is at least the logarithmic derivative of the right hand side for every $y\in(0,1)$. After taking logarithmic derivative of both sides and simplifying by $-y$ it remains to show that
\begin{equation}\label{eq2}
\frac{a(a-1)}{1+(a-1)y}\leq\sum_{i=1}^b\frac{d_i(d_i-1)}{1+(d_i-1)y}.
\end{equation}
Now using $d_i\leq a$ we have that $$\frac{d_i(d_i-1)}{1+(d_i-1)y}\geq \frac{d_i(d_i-1)}{1+(a-1)y}$$ holds for $1\leq i\leq b$ and this together with $\sum_{i=1}^b d_i(d_i-1)\geq a(a-1)$ and (\ref{eq2}) finishes the proof.
\end{proof}

\section{the recoupling algorithm}\label{alg}

We finish this paper with an algorithm that we used in computer experiments to approximate the determinant maximizing matrices $\Sigma(G,x)$. The algorithm can also be used to compute the coefficients in the power series expansion of $\tau(G,x)$.

Let $G=(V,E)$ be a fixed graph and $x\in [0,1)$. Let $M_0(G,x)$ denote the $V\times V$ matrix with $1's$ in the diagonal and $x$ elsewhere. Our algorithm produces a sequence of matrices $M_i(G,x)$ recursively with increasing determinants such that they converge to $\Sigma(G,x)$. 

\noindent{\bf the recursive step} {\it To produce $M_{i+1}(G,x)$ from $M_i(G,x)$ we choose a non-edge $e_{i+1}:=(v,w)\in V\times V$ with $e_{i+1}\notin E$. Let $A:= V\setminus{v}$ and $B:=V\setminus{w}$. Then we set $$M_{i+1}(G,x):=M_i(G,x)_{A\times A}\curlyvee M_i(G,x)_{B\times B}.$$}

The algorithm depends on a choice of non edges $e_1,e_2,\dots$. Our choice is to repeat a fix ordering of all non-edges several times. One can also perform the algorithm with formal matrices in which the entries are rational functions of $x$. It is easy to see by induction that in each step the entries remain of the form $f(x)/(1+xg(x))$ for some polynomials $f,g\in\mathbb{Z}(x)$. This implies that the powers series expansions of the entries have integer coefficients. These coefficients stabilize during the algorithm and this provides a method to compute the power series of $\tau(G,x)$ around $0$. The formulas for the coefficients will be given in a subsequent paper.

\subsection*{Acknowledgement.} The research leading to these results has received funding from the European Research Council under the European Union's Seventh Framework Programme (FP7/2007-2013) / ERC grant agreement n$^{\circ}$617747. The research was partially supported by the MTA R\'enyi Institute Lend\"ulet Limits of Structures Research Group.

\end{document}